\documentclass[ECP,preprint]{ejpecp} 

\SHORTTITLE{Distribution of the roots of Eulerian polynomials}

\TITLE{Distribution of the roots of Eulerian polynomials} 

\AUTHORS{%
  Paul~Melotti\footnote{Université Paris-Saclay
    \EMAIL{paul.melotti@universite-paris-saclay.fr}}\orcid{0000-0002-2004-2764}}


\AMSSUBJ{60B10} 
\AMSSUBJSECONDARY{26C10 ; 60C05 ; 11B68 ; 05A15} 

\ABSTRACT{We give a new proof that the empirical measures of the roots of Eulerian polynomials converge to a  certain log-Cauchy distribution. To do so, we show that each moment of the roots of a related family of polynomials not only converge, but in fact become ultimately constant. These asymptotic moments are expressed in terms of Cauchy numbers of the second kind.}

\newcommand{\C}{\mathbb{C}}

\begin{document}

\section{Introduction}

For $1\leq k \leq n$, let $A(n,k)$ denote the Eulerian number, that is the number of permutations of size $n$ with exactly $k-1$ descents. The Eulerian polynomials are 
\begin{equation*}
	A_n(x) = \sum_{k=1}^n A(n,k)x^k.
\end{equation*}
It is a classical result that this polynomial has $n$ distinct real roots, see for instance~\cite[Theorem~1.34]{Bona}. Let us denote these roots by $x_{n,1} < x_{n,2} < \dots < x_{n,n} = 0$. We are interested in their empirical measure, for which we consider $-x_{n,k}$ instead, so as to work with positive values:
\begin{equation*}
	\mu_n = \frac{1}{n} \sum_{k=1}^n \delta_{-x_{n,k}}.
\end{equation*}
The asymptotic behaviour of $\mu_n$ was described by Sobolev \cite{Sobolev} and Sirazhdinov \cite{Sirazhdinov}, who used analytic methods to get the following theorem. In this paper we give a new, combinatorial proof of this result.
\begin{theorem}[\cite{Sobolev,Sirazhdinov}]
	\label{theo:distri}
	As $n\to \infty$, the sequence of measures $\mu_n$ converges weakly to a probability measure $\mu$ with support $[0,\infty)$. This measure is the distribution of $\exp(\pi Z)$ where $Z$ is a standard Cauchy random variable. That is, $\mu$ has density
	\begin{equation*}
		\frac{1}{t \left(\pi^2 + \log^2 t\right)} 1_{t>0}.
	\end{equation*}
\end{theorem}

\begin{figure}[ht]
    \centering
    \includegraphics[scale=0.6]{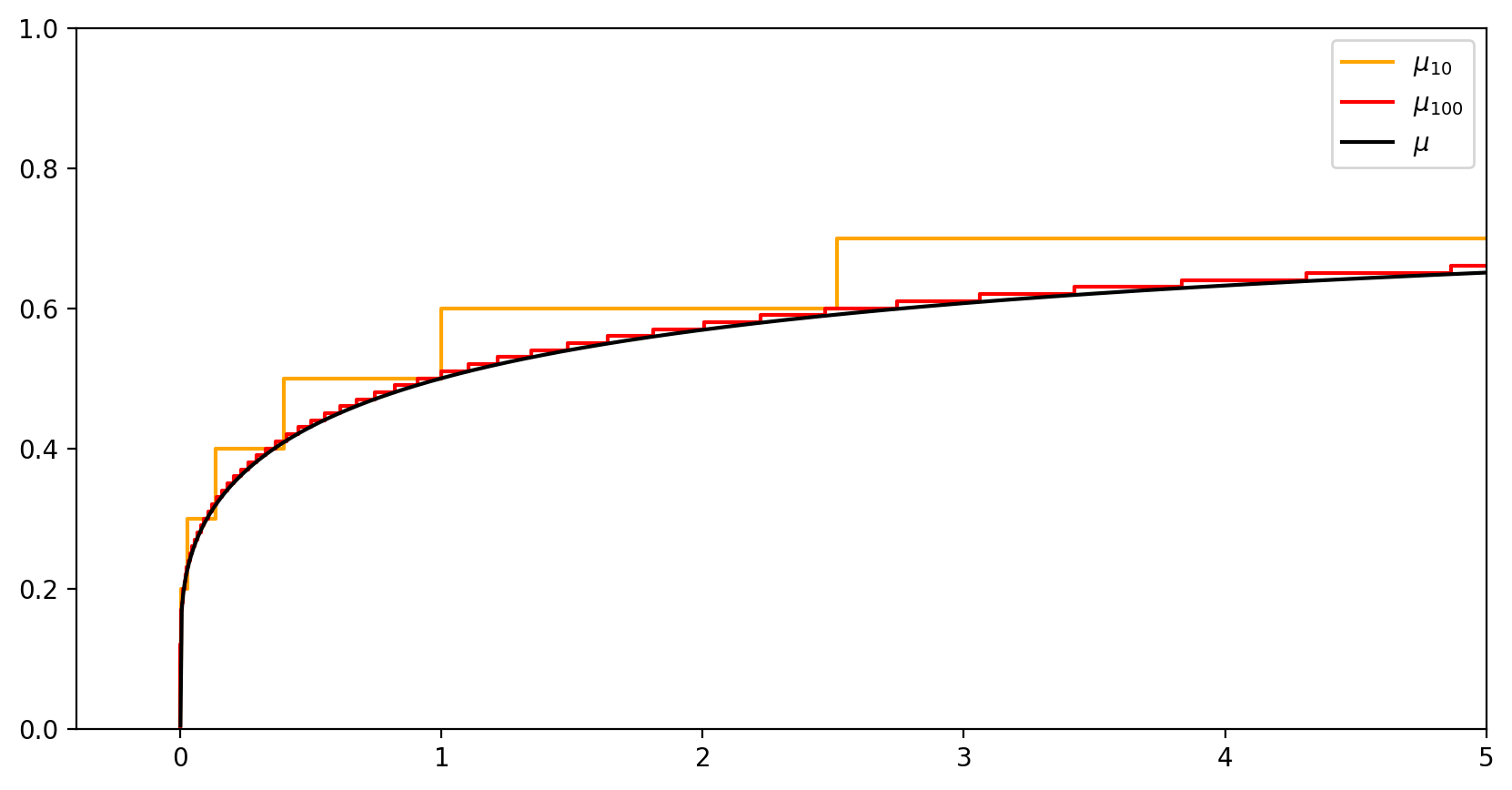}
	\caption{Cumulative distribution functions of $\mu_{10}, \mu_{100}$ and the limiting measure $\mu$.}
	\label{fig:mu}
\end{figure}

Finding the asymptotic distribution of a family of (random or deterministic) zeros is a common problem in random matrix theory, see~\cite{AGZ}, from which we borrow the method of moments and Stieltjes transforms. In a recent series of preprints~\cite{JKM1,JKM2}, Jalowy, Kabluchko and Marynych develop another method to find the limiting distribution of zeros, based on the asymptotic behaviour of coefficients for many families of polynomials, of which Eulerian polynomials are an example. Our method differs from theirs in that we rely on exact computations of moments, as in Theorem~\ref{theo:moments} below, which we believe gives a short and interesting way to prove Theorem~\ref{theo:distri}. Asymptotics of coefficients have also been used in the case of some Stirling polynomials by Elbert~\cite{Elbert}, and for some orthogonal polynomials by Lubinsky and Sidi~\cite{LS1,LS2}.

The method of moments consists in computing the moments of the sequence of measures, proving that they converge towards the moments of the limiting measure, and using a unicity argument. However, we cannot apply it directly in our case, as all moments diverge\footnote{This can be seen by noting that $\sum_{k=1}^n x_{n,k} = - A_{n,n-1} = -2^n + n + 1$, which does not scale like $n$.}. This problem can be avoided by studying instead
\begin{equation*}
	u_{n,k} = \frac{1}{1-x_{n,k}} \in (0,1]
\end{equation*}
and by computing the moments of the empirical measure of the $u_{n,k}$ instead. These moments (appropriately rescaled) happen to have a remarkable, closed expression. To state it, we denote by $C_p$ the $p$-th Cauchy number of the second kind, which can be defined as
\begin{equation*}
	C_p = \int_0^1 x (x+1) \dots (x+p-1) dx,
\end{equation*}
see for instance \cite[Definition~5.3.1]{Mezo}. The numbers $(-1)^p C_p$ are also known as Nörlund numbers~\cite{Howard}.

\begin{theorem}
	\label{theo:moments}
	For any $1\leq p \leq n$,
	\begin{equation*}
		\frac{1}{n+1} \sum_{k=1}^n u_{n,k} ^p = \frac{C_p}{p!}.
	\end{equation*}
\end{theorem}
We emphasize that the right-hand side does not depend on $n$. The choice of normalizing the moments by $n+1$ instead of $n$ comes from this remarkable property. The first values of this moment sequence are $\frac{1}{2}, \frac{5}{12}, \frac{3}{8}, \frac{251}{720}, \frac{95}{288}, \dots$ 

\begin{remark}
	The numbers $u_{n,k}$ are also the opposite of the roots of a related family of polynomials, the Fubini polynomials, which count ordered set partitions by number of blocks; see~\cite[Subsection~6.2]{Mezo}. The asymptotic version of Theorem~\ref{theo:moments} can be found in~\cite[Theorem~4.4]{JKM2}. In particular, by the method of moments, Theorem~\ref{theo:moments} implies that there is a limiting measure for the empirical measures of roots of Fubini polynomials, whose Stieljes transform we also compute in Section~\ref{sec:pr1}.
\end{remark}

\begin{remark}
	Formal computations suggest that when $n$ is even, Theorem~\ref{theo:moments} also holds for $p=n+1$. Apart from these, it seems that the identity fails for $p>n+1$ in the even case, and for $p>n$ in the odd case.
\end{remark}

In Section~\ref{sec:symu}, we proceed with the proof of Theorem~\ref{theo:moments}, which relies on computing the symmetric functions of the roots $u_{n,k}$ in terms of Stirling numbers of the second kind, and a relation between Stirling numbers and Cauchy numbers. In Section~\ref{sec:pr1} we find the asymptotic distribution of the $u_{n,k}$ \emph{via} its Stieljes transform, and deduce that of the $x_{n,k}$, proving Theorem~\ref{theo:distri}.

\section{Symmetric functions and moments of the $u_{n,k}$}
\label{sec:symu}

For $0\leq p \leq n$, we denote the elementary symmetric functions of the $\left(u_{n,k}\right)_{k=1}^n$ by
\begin{equation*}
	e_{n,p} = \sum_{1\leq i_1 < \dots < i_p \leq n} u_{n,i_1} \cdots u_{n,i_p}
\end{equation*}
and the convention $e_{n,0}=1$.

Let $S(n,k)$ denote the Stirling numbers of the second kind, that is the number of set partitions of $\{1,\dots, n\}$ into $k$ blocks.

\begin{lemma}
	\label{le:symu}
	For any $0\leq p \leq n$,
	\begin{equation}
		\label{eq:QG}
		e_{n,p} = \frac{(n-p)!}{n!} S(n+1,n-p+1).
	\end{equation}
\end{lemma}

\begin{proof}
	Since $A_n$ has roots at $x_{n,k}$ and is monic,
	\begin{equation*}
		\begin{split}
			A_n(1-x) & = \prod_{k=1}^n (1-x-x_{n,k}) \\
			& = \prod_{k=1}^n \left(\frac{1}{u_{n,k}}-x\right) \\
			& = n! \prod_{k=1}^n (1-xu_{n,k}) \\
			& = n! \sum_{p=0}^n (-1)^p e_{n,p}x^p ,
		\end{split}
	\end{equation*}
	where we used that the constant term is $A_n(1)=n!$. On the other hand,
	\begin{equation*}
		\begin{split}
			A_n(1-x) & = \sum_{k=1}^n \sum_{p=0}^k (-1)^p \binom{k}{p} A(n,k) x^p \\
			& = \sum_{p=0}^n (-1)^p x^p \sum_{k=p}^n \binom{k}{p} A(n,k) 
		\end{split}
	\end{equation*}
	The inner sum can be rewritten, using the symmetry $A(n,k)=A(n,n+1-k)$, so that it becomes
	\begin{equation*}
	    \sum_{l=1}^{n+1-p} \binom{n+1-l}{n+1-p-l} A(n,l).
	\end{equation*}
	Using the relation between Eulerian numbers and Stirling numbers of the second kind, see for instance \cite[Theorem~1.18]{Bona} (which can be applied after using Pascal's formula), this becomes $(n+1-p)! S(n,n-p+1) + (n-p)! S(n,n-p)$. By the standard recursion of Stirling numbers, this is also $(n-p)! S(n+1,n-p+1)$. Therefore,
	\begin{equation*}
		A_n(1-x) = \sum_{p=0}^n (-1)^p (n-p)! S(n+1,n-p+1) x^p.
	\end{equation*}

	Identifying the term in $x^p$ in both expressions yields the lemma.	
\end{proof}

\begin{remark}
	The expression for $A_n(1-x)$ in terms of Stirling numbers can also be obtained from its connection with Fubini polynomials, see for instance~\cite[(6.22)]{Mezo}.
\end{remark}

The proof of Theorem~\ref{theo:moments} relies on Newton identities, Lemma~\ref{le:symu}, and a relation between Cauchy numbers and Stirling numbers, which we state now. We use the common convention $C_0=1$.
\begin{lemma}
	\label{le:st_n}
	For any $1\leq p \leq n$,
	\begin{equation*}
		\sum_{i=0}^{n-p} (-1)^i \binom{p+i-1}{i} S(n,p+i) C_i = \frac{p}{n} S(n,p).
	\end{equation*}
\end{lemma}

\begin{proof}
	Let us start with the mixed bivariate generating function of the Stirling numbers of the second kind~\cite[(1.94b)]{Stanley}:
	\begin{equation}
		\label{eq:ser_st}
		\sum_{p\geq 0} \sum_{n\geq p} S(n,p) \frac{x^n}{n!} y^p = \exp\left(y(e^x-1)\right)
	\end{equation}
	with the convention that $S(n,0)=0$ for $n\geq 1$, and $S(0,0)=1$. Differentiating with respect to $y$ and then multiplying by $y$, we get
	\begin{equation}
		\label{eq:ser_lhs}
		\sum_{p\geq 1} \sum_{n\geq p} \frac{p}{n} S(n,p) \frac{x^n}{(n-1)!} y^p = y (e^x-1) \exp\left(y(e^x-1)\right)
	\end{equation}
	which is a mixed generating function of the right-hand side of the lemma. Let us compute the same function of the left-hand side, and change index $i$ into $j=p+i$:
	\begin{equation*}
		\begin{split}
			& \sum_{p\geq 1}\sum_{n\geq p}\sum_{i=0}^{n-p} (-1)^i  \binom{p+i-1}{i} S(n,p+i) C_i \frac{x^n}{(n-1)!} y^p \\
			= & \sum_{p\geq 1} \sum_{j\geq p} (-1)^{j-p}  \binom{j-1}{j-p} C_{j-p} y^p \sum_{n\geq j} S(n,j)\frac{x^n}{(n-1)!}.
		\end{split}
	\end{equation*}
	From \eqref{eq:ser_st}, we can also extract $ \sum_{n\geq j} S(n,j)\frac{x^n}{n!} = \frac{(e^x-1)^j}{j!}$, which by differenting $x$ and multiplying by $x$ gives the value of the inner sum. Thus the previous expression becomes
	\begin{equation*}
		\begin{split}
			& x e^x \sum_{p\geq 1} \sum_{j\geq p} (-1)^{j-p}  \binom{j-1}{j-p} C_{j-p} y^p \frac{(e^x-1)^{j-1}}{(j-1)!} \\
			= & x y e^x \left( \sum_{q\geq 0}(-1)^q  C_q \frac{(e^x-1)^q}{q!} \right) \left( \sum_{r \geq 0} \frac{(y(e^x-1))^r}{r!} \right)
		\end{split}
	\end{equation*}
	where we have set $q=j-p$ and $r=p-1$. The exponential generating function of Cauchy numbers of the second kind is known to be~\cite[(5.16)]{Mezo}
	\begin{equation}
		\label{eq:norlund_gf}
		\sum_{q\geq 0} (-1)^q C_q \frac{t^q}{q!} = \frac{t}{(1+t)\log(1+t)}.
	\end{equation}
	Injecting into the previous expression, we get
	\begin{equation*}
		xye^x \frac{e^x-1}{xe^x} \exp\left(y(e^x-1)\right) = y(e^x-1) \exp\left(y(e^x-1)\right)
	\end{equation*}
	which is the same as \eqref{eq:ser_lhs}.
\end{proof}

We now have all the elements to prove Theorem~\ref{theo:moments}.
\begin{proof}[Proof of Theorem~\ref{theo:moments}]
	For any fixed $n\geq 1$, we proceed by induction over $p$. Throughout the proof let us drop the $n$ from the subscript notations $e_{n,p}$, $u_{n,p}$ etc., and also denote $m_p = \sum_{k=1}^n u_{n,k}^p$.
	
	For $p=1$, first note that the Eulerian numbers satisfy $A(n,n+1-k)=A(n,k)$, which implies that for any nonzero root $x_{k}$, $\frac{1}{x_{k}}$ is also a root of $A_n$. Using this involution, and the convention $x_{n,n}=0$ so $u_{n,n}=1$, we get
	\begin{equation*}
		m_1 = 1 + \sum_{k=1}^{n-1} \frac{1}{1-x_{k}} = 1 + \sum_{k=1}^{n-1} \frac{1}{1-\frac{1}{x_{k}}} = 1 + \sum_{k=1}^{n-1} \frac{-x_{k}}{1-x_{k}}.
	\end{equation*}
	By summing the second and last expressions together, we obtain $2m_1 = n+1$ and therefore $\frac{1}{n+1}m_1 = \frac12 = C_{1}$.
	
	Now for $2 \leq p \leq n$, suppose that the formula holds for all values $1,\dots, p-1$. By Newton's identities\footnote{These identities would take a slightly different form for $p>n$, so we use the hypothesis here. Also note that one would need a more subtle version of Lemma~\ref{le:st_n} in the case $p>n$ to take into account the change of indices.}, then Lemma~\ref{le:symu} and the induction hypothesis,
	\begin{equation}
		\label{eq:newtmp}
		\begin{split}
			m_p  = & (-1)^{p-1} p e_p + \sum_{i=1}^{p-1}  (-1)^{p-1+i} e_{p-i} m_i \\
			= & (-1)^{p-1} p \frac{(n-p)!}{n!} S(n+1,n-p+1) \\
			& + (-1)^{p-1} (n+1) \sum_{i=1}^{p-1} (-1)^i \frac{(n-p+i)!}{n!} S(n+1,n-p+i+1) \frac{C_i}{i!}.
		\end{split}
	\end{equation}
	For the last sum, we use Lemma~\ref{le:st_n} with $n$ replaced by $n+1$ and $p$ by $n-p+1$, which gives
	\begin{equation*}
		\sum_{i=0}^{p} \binom{n-p+i}{i} (-1)^i S(n+1,n-p+i+1) C_i = \frac{n-p+1}{n+1} S(n+1,n-p+1)
	\end{equation*}
	or, removing the two extremal indices,
	\begin{equation*}
		\sum_{i=1}^{p-1} \binom{n-p+i}{i} (-1)^i S(n+1,n-p+i+1) C_i = - \frac{p}{n+1} S(n+1,n-p+1) - (-1)^p \binom{n}{p} C_p.
	\end{equation*}
	Injecting this into \eqref{eq:newtmp}, we have
	\begin{equation*}
		\begin{split}
			m_p = & (-1)^{p-1} p \frac{(n-p)!}{n!} S(n+1,n-p+1) \\ 
			& + (-1)^{p-1} (n+1) \frac{(n-p)!}{n!} \left( - \frac{p}{n+1} S(n+1,n-p+1) - (-1)^p \binom{n}{p} C_p \right) \\
			= & (n+1) \frac{C_p}{p!}
		\end{split}
	\end{equation*}
	which concludes the induction step.
\end{proof}

\section{Weak convergence and limiting distribution}
\label{sec:pr1}

We can now work towards the proof of Theorem~\ref{theo:distri}.

Let $\nu_n$ be the probability measure
\begin{equation*}
	\nu_n = \frac{1}{n} \sum_{k=1}^n \delta_{u_{n,k}}.
\end{equation*}
As a direct consequence of Theorem~\ref{theo:moments}, for any $p\geq 1$ the $p$-th moments of $\nu_n$ converge as $n \to \infty$ towards $\frac{C_p}{p!}$. Since the measures are all supported in the compact set $[0,1]$, they are tight and the only subsequential weak limit is the unique probability measure $\nu$ with moments
\begin{equation*}
	\forall p \geq 1, \ \int_0^1 u^p \nu(du) = \frac{C_p}{p!}.
\end{equation*}
As a result, $\nu_n$ converges weakly towards this measure $\nu$, see \cite[3.3.5]{Durrett}.

This allows for other characterizations of $\nu$, for instance via its Stieltjes transform, which can be directly computed from \eqref{eq:norlund_gf}:
\begin{equation}
	\label{eq:stieljes_nu}
	\forall t\in \C\setminus[0,1], \ S_\nu(t) = \int_0^1 \frac{1}{u-t} \nu(du) = \frac{1}{t(t-1)\log\left(1-\frac{1}{t}\right)}.
\end{equation}
where we use the principal value of the logarithm, which gives an analytic function in the domain.

Moreover, since $-x_{n,k} = \frac{1}{u_{n,k}}-1$, by the continuous mapping theorem we also get that the sequence of measures $\mu_n$ converges weakly towards $\mu$, where $\mu$ is the pushforward measure of $\nu$ by the map $u \mapsto \frac{1}{u}-1$. We can also find its Stieltjes transform by using \eqref{eq:stieljes_nu}, substitution and direct computations, which leads to
\begin{equation*}
	\begin{split}
		\forall t\in \C\setminus [0,\infty), \ S_\mu(t) & = \int_{0}^{\infty} \frac{1}{x-t} \mu(dx) \\
		& = \int_{0}^{1} \frac{1}{\frac{1}{u}-1-t} \nu(du) \\
		& = \frac{1}{t \log (-t)} - \frac{1}{t+1}.
	\end{split}
\end{equation*}
From there, by the inverse Stieltjes transform procedure, see for instantce~\cite[Theorem~2.4.3]{AGZ}, we get that for any interval $I\subset [0,\infty)$,
\begin{equation*}
	\mu(I) = \lim_{\epsilon \to 0}  \frac{1}{\pi} \int_I \Im \left( S_\mu(\lambda + i\epsilon)\right) d\lambda.
\end{equation*}
For the chosen logarithmic branch, we have $\log(-\lambda-i\epsilon) = \log\lambda - i\pi + O(\epsilon)$, from which we get $S_\mu(\lambda + i\epsilon) = \frac{\log\lambda \ + \ i\pi}{\lambda\left(\log^2\lambda \ + \ \pi^2\right)} - \frac{1}{\lambda+1} + O(\epsilon) $, with a $O$ constant uniform in $\lambda$ for $\lambda$ bounded and bounded away from $0$. Therefore, for $I=[a,b]$ with $0<a<b$, we may switch the limit and the integral, which gives
\begin{equation*}
	\mu([a,b]) = \int_a^b \frac{1}{\lambda\left(\log^2\lambda + \pi^2\right)} d\lambda
\end{equation*}
and we can extract the density of $\mu$. The fact that this is also the distribution of $\exp(\pi Z)$, where $Z$ is a standard Cauchy random variable, is a direct computation. This concludes the proof of Theorem~\ref{theo:distri}.

\begin{acks}
	I thank the anonymous referee for many interesting suggestions and remarks.	
\end{acks}



\providecommand{\bysame}{\leavevmode\hbox to3em{\hrulefill}\thinspace}
\providecommand{\MR}{\relax\ifhmode\unskip\space\fi MR }
\providecommand{\MRhref}[2]{%
	\href{http://www.ams.org/mathscinet-getitem?mr=#1}{#2}
}
\providecommand{\href}[2]{#2}

\end{document}